\documentclass[11pt]{amsart}
\usepackage{graphicx}
\usepackage{amssymb}
\usepackage{amsmath}
\usepackage{amsthm}
\usepackage{amscd}
\usepackage{amsfonts,epsfig,latexsym}
\usepackage{todonotes}
\def\R{\mathbb{R}}
\def\N{\mathbb{N}}

\def\Z{\mathbb{Z}}

\def\l{\lambda}
\def\pwt{\phi_{\omega,\tau}}

\newtheorem{Proposition}{Proposition}[section]
\newtheorem{Remark}{Remark}
\newtheorem{Lemma}{Lemma}[section]

\newtheorem{Theorem}{Theorem}[section]
\newtheorem*{Theorem*}{Theorem}
\newtheorem*{Conjecture*}{Conjecture}
\newtheorem{Corollary}{Corollary}[section]

\begin{document}

\title{On Sums of Nearly Affine Cantor Sets}

\author{{A.~Gorodetski}}
\address{Anton~Gorodetski
\newline\hphantom{iii} University of California, Irvine}
\email{asgor@math.uci.edu}
\thanks{A.\ G.\  and S.\ N.\ were supported in part by NSF grants  DMS--1301515.}

\author{{S. Northrup}}
\address{Scott Northrup
\newline\hphantom{iii} University of California, Irvine}
\email{snorthru@math.uci.edu}

\subjclass[2010]{Primary: 28A80, 37D99, 28A78, }

\keywords{Cantor sets, arithmetic sums, Hausdorff Dimension}

\begin{abstract}
For a compact set $K\subset \mathbb{R}^1$ and a family $\{C_\lambda\}_{\lambda\in J}$ of dynamically defined Cantor sets sufficiently close to affine with $\text{dim}_H\, K+\text{dim}_H\, C_\lambda>1$ for all $\lambda\in J$, under natural technical conditions we prove that the sum $K+C_\lambda$ has positive Lebesgue measure for almost all values of the parameter $\lambda$. As a corollary, we show that generically the sum of two affine Cantor sets has positive Lebesgue measure provided the sum of their Hausdorff dimensions is greater than one.
\end{abstract}

\maketitle

\section{Introduction and Main results}\label{sec:intro}
Questions on the structure and properties of sums of Cantor sets appear naturally in dynamical systems \cite{n1, n2, n3, PaTa}, number theory \cite{CF, Mal, Moreira}, harmonic analysis \cite{BM, BKMP}, and spectral theory \cite{EL06, EL07, EL08, Y}. J.\,Palis asked whether it is true (at least generically) that the arithmetic sum of dynamically defined Cantor sets either has measure zero, or contains an interval (see \cite{PaTa}). This claim is currently known as the ``Palis' Conjecture''. The conjecture was answered affirmatively in \cite{MY} for generic dynamically defined Cantor sets. For sums of generic {\it affine} Cantor sets Palis' Conjecture is still open.

Even for the simplest case of middle-$\alpha$ Cantor sets these questions are non-trivial and not completely settled. By a middle-$\alpha$ Cantor set we mean the Cantor set
\begin{equation}\label{e.ddcs}
C=\cap_{n=0}^{\infty}I_n,  \ I_{n+1}=\cup_{i=1}^{m}\varphi_i(I_n), \ \varphi_i(I_0)\cap \varphi_j(I_0)=\emptyset \ \text{for}\ i\ne j,
\end{equation}
where $I_0=[0,1]$, $m=2$, $\varphi_1(x)=ax$, $\varphi_2(x)=(1-a)+ax$, $a=\frac{1}{2}(1-\alpha)$. Let us denote this Cantor set by $C_a$.

It is easy to show (using dimensional arguments, e.g. see Proposition 1 in Section 4 from \cite{PaTa}) that if $\frac{\log 2}{\log 1/a}+\frac{\log 2}{\log 1/b}<1$ then $C_a+C_b$ is a Cantor set. On the other hand, Newhouse's Gap Lemma (e.g. see Section 4.2 from \cite{PaTa}, or \cite{n1}) implies that if $\frac{a}{1-2a}\frac{b}{1-2b}>1$ then $C_a+C_b$ is an interval. This still leaves a ``mysterious region'' $R$ in the space of parameters, see Figure \ref{f.1}, and Solomyak \cite{So97} showed that for Lebesgue a.e. $(a, b)\in R$ one has $Leb(C_a+C_b)>0$.
 \begin{figure}[h]\label{f.1}
\begin{center}
\includegraphics[height=2in]{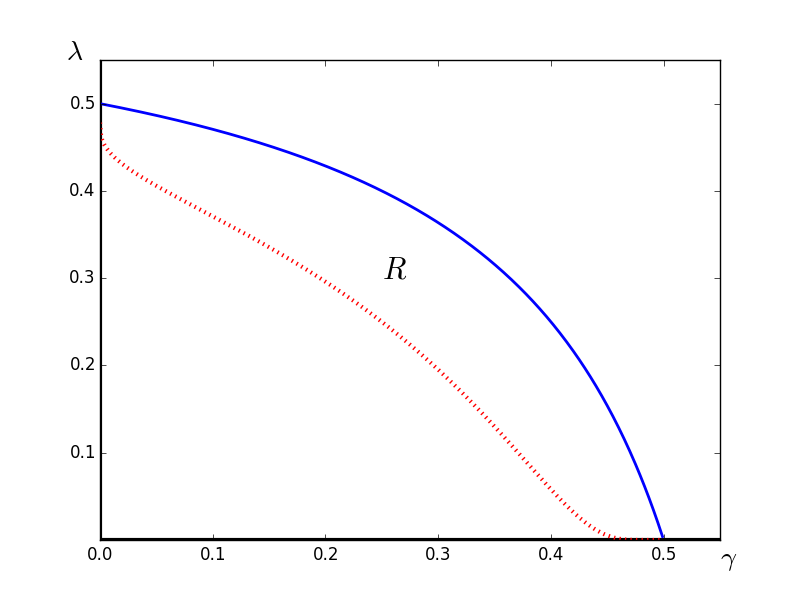} \caption{The region $R$ studied by B.\,Solomyak in \cite{So97}}
\end{center}
\end{figure}
 A description of possible topological types of $C_a+C_b$ was provided in \cite{MO}. It is still an open question whether $C_a+C_b$ contains an interval for a.e. $(a,b)\in R$.

 Solomyak's result was generalized to families of homogeneous self-similar Cantor sets (i.e. Cantor sets given by (\ref{e.ddcs}) where all contractions $\{\varphi_i\}_{i=1, \ldots, m}$ are linear with the same contraction coefficient) by  Peres  and Solomyak \cite{PeSo}. They showed that for a fixed compact set $K\subseteq \mathbb{R}$ and a family $\{C_\lambda\}$ of homogeneous Cantor sets parameterized by a contraction rate $\lambda$ (i.e. all  contractions have the form $\varphi_i(x)=\lambda x +D_i(\lambda)$, $D_i\in C^1$) one has
 \begin{multline}\label{e.hdsone}
     \text{dim}_H(C_\lambda+K)= \text{dim}_HC_\lambda +  \text{dim}_HK \ \text{for a.e.}\ \lambda\in (\lambda_0, \lambda_1) \\ \text{if} \ \ \ \text{dim}_HC_\lambda +  \text{dim}_HK<1\ \  \text{for all}\ \lambda\in (\lambda_0, \lambda_1), \ \text{and}
 \end{multline}
 \begin{multline}\label{e.hdlone}
     Leb(C_\lambda+K)>0 \ \text{for a.e.}\ \lambda\in (\lambda_0, \lambda_1) \\ \text{if} \ \ \ \text{dim}_HC_\lambda +  \text{dim}_HK>1\ \  \text{for all}\ \lambda\in (\lambda_0, \lambda_1).
 \end{multline}

 In the case when $K$ is a non-linear $C^{1+\varepsilon}$-dynamically defined Cantor set, the set of exceptional parameters in (\ref{e.hdlone}) in fact has zero Hausdorff dimension, see \cite[Theorem 1.4]{Shm}.

 For a more general case of sums of dynamically defined Cantor sets $C$ and $K$ on the first glance the mentioned above results by Moreira and Yoccoz \cite{MY} provide the complete answer. But in practice in many cases one has to deal with a finite parameter families of Cantor sets, or even with a specific fixed Cantor sets $C$ and $K$, and \cite{MY} does not provide specific genericity assumptions that could be verified in a particular given setting. Specific conditions that would allow to claim that
 \begin{equation}\label{e.hs}
 \text{dim}_H(C+K)=\min\left(\text{dim}_HC +  \text{dim}_HK, 1\right)
 \end{equation}
 are currently known \cite{HS,NPS, PeShm}, but the case $\text{dim}_HC +  \text{dim}_HK>1$ turned out to be more subtle.

 In this paper we address this question in the case of affine (all $\varphi_i$ in (\ref{e.ddcs}) are affine contractions, not necessarily with the same contraction coefficients) and close to affine dynamically defined Cantor sets.
 \begin{Theorem}\label{t.1}
Suppose $J\subseteq \mathbb{R}$ is an interval and $\lbrace C_\l \rbrace_{\lambda\in J}$ is  a family of dynamically defined Cantor sets generated by contracting maps \begin{equation}\label{e.start}\lbrace f_{i, \lambda}(x)=c_{i}(\lambda)x+b_{i}(\lambda)+g_{i}(x, \lambda)\rbrace_{i=1}^m\end{equation} such that the following holds:
\begin{equation}
c_i(\lambda), b_i(\lambda)\ \text{are $C^1$-functions of}\ \lambda;
\end{equation}
\begin{equation}\label{e.c}
\frac{d|c_i|}{d\lambda} \le -\delta<0\  \text{for all $\lambda\in J$ and some uniform $\delta>0$};
\end{equation}
\begin{equation}\label{e.g}
g_i(x,\lambda)\in C^2, \ \text{with small (based on $\{c_i(\lambda)\}, \{b_i(\lambda)\}$) $C^2$-norm}.
\end{equation}
Then for any compact $K\subset \mathbb{R}$ with
\begin{equation}\label{e.condition}
\dim_H(K)+\dim_H(C_{\lambda}) > 1\ \text{ for all }\ \l\in J,
\end{equation}
the sumset $K+C_\lambda$ has positive Lebesgue measure for a.e. $\l\in J$.
\end{Theorem}

\begin{Remark}
Theorem \ref{t.1} can be generalized in a straightforward way to a larger class of nearly affine Cantor sets where topological Markov chains are allowed instead of the full Bernoulli shift in the symbolic representation (see \cite{MY} or \cite{PaTa} for detailed definitions). We restrict ourselves to the case of the full shift only to keep the exposition more transparent.
\end{Remark}

\begin{Remark}
We strongly believe that the assumption on $C_\lambda$ being close to affine is an artefact of the proof, and that a similar statement should hold in a more general setting, for a family of non-linear dynamically defined Cantor sets without any smallness assumptions on non-linearity. We plan to address this question in a future publication.
\end{Remark}

Consider now the non-homogeneous affine case, that is a Cantor set generated by (\ref{e.ddcs}), where $\varphi_k(x)=\lambda_kx+d_k$. Moreover, let us include it into a family $\{K_{\Lambda}\}$, where
\begin{equation}
\Lambda=(\lambda_1, \ldots, \lambda_m), \ \lambda_k\in J_k\subset (-1, 0)\cup (0,1), \ \text{and} \ d_k=d_k(\Lambda)\ \text{is $C^1$}.
\end{equation}
The last condition in (\ref{e.ddcs}) implies that
$$
[d_i(\Lambda)+\l_i K_\Lambda]\cap [d_j(\Lambda)+\l_j K_\Lambda] = \emptyset\ \text{ for}\  i\neq j,
$$
which is sometimes called {\it strong separation condition}, e.g. see \cite{PeSo}.

%
%

Fubini's theorem together with Theorem \ref{t.1} gives the following statement.

\begin{Corollary}\label{c.1}
Suppose $K$ is a compact subset of the real line,  and a family $\{K_\Lambda\}$ of affine Cantor sets as above is given such that 
$$
\dim_H K + \dim_H K_\Lambda > 1 \ \text{ for all}\  (\l_1,\dots,\l_m)\in J_1\times\dots\times J_m.
$$
 Then for a.e. $(\l_1,\dots,\l_m)\in J_1\times\dots\times J_m$ the set $K+K_\Lambda$ has positive Lebesgue measure.
\end{Corollary}

Notice that in this setting any affine Cantor set is completely determined by $2m$ parameters, namely $(\lambda_1, \ldots, \lambda_m)\in \Lambda$ and $(d_1, \ldots, d_m)$. Admissible $2m$-tuples of the parameters (i.e. such that $\varphi_k([0,1])\subseteq [0,1]$ for each $k=1, 2, \ldots, m$, and the strong separation condition holds) form a subset in $\mathbb{R}^{2m}$. As an immediate consequence of Corollary  \ref{c.1} we have

\begin{Corollary}\label{c.2}
Generically (for Lebesgue almost all admissible tuples of the parameters) the sum of two affine Cantor sets has positive Lebesgue measure provided the sum of their Hausdorff dimensions is greater than one.
\end{Corollary}

\begin{Remark}
It is interesting to compare these results with Theorem E from \cite{ShmS} that claims that for any two affine Cantor sets $C_1$ and $C_2$ with sum of dimensions greater than one, $\text{dim}_H\,\{u\in \mathbb{R}\ |\ Leb(C_1+uC_2)=0\}=0$.
\end{Remark}

The idea of proof of Theorem \ref{t.1} is to find some measures supported on $K$ and $C_\lambda$ whose convolution is absolutely continuous with respect to the Lebesgue measure. Since support of a convolution of two measures is the sum of their supports, this would  prove that $Leb(K+C_\lambda)>0$. In Section \ref{s.acc} we provide the statement of a result from \cite{DGS} on absolute continuity of convolutions of singular measures under certain conditions. Then in Section \ref{sec:main} we verify those conditions for some specific measures supported on $K$ and $C_\lambda$.

\section{Absolute continuity of convolutions}\label{s.acc}

Let $\Omega=\mathcal{A}^{\Z_+}$ with $|\mathcal{A}| = m \geq 2$ be the standard symbolic space, equipped with the product topology. Let $\mu$ be a Borel probability measure on $\Omega$.

Let $J$ be a compact interval and assume we are given a family of continuous maps $\Pi_\l:\Omega\to\R$, for $\l \in J$, such that $C_\l = \Pi_\l(\Omega)$ are the Cantor sets, and let $\nu_\l = \Pi_\l (\mu)$.

For a word $u\in\mathcal{A}^{n}$, $n\geq 0$, denote by $|u| = n$ its length and by $[u]$ the cylinder set of elements of $\Omega$ that have $u$ as a prefix.  For $\omega,\tau\in\Omega$ we write $\omega\wedge\tau$ for the maximum common subword in the beginning of $\omega$ and $\tau$ (empty if $\omega_0\neq\tau_0$; we set the length of the empty word to be zero).  For $\omega,\tau\in \Omega$, let $\pwt(\l):=\Pi_\l(\omega)-\Pi_\l(\tau)$.

We will need the following statement.

\begin{Proposition}[Proposition 2.3 from \cite{DGS}]\label{BlackBox}
Let $\eta$ be a compactly supported Borel probability measure on $\R$ of exact local dimension $d_\eta$.  Suppose that for any $\varepsilon > 0$ there exists a subset $\Omega_\varepsilon \subset \Omega$ such that $\mu(\Omega_\varepsilon) > 1 -\varepsilon$ and the following holds;  there exist constants $C_1,C_2,C_3,\alpha,\beta,\gamma>0$ and $k_0\in\Z_+$ such that
\begin{equation}\label{BlackBox0}
d_\eta+\dfrac{\gamma}{\beta} > 1  \text{ and } d_\eta > \dfrac{\beta -\gamma}{\alpha},
\end{equation}
\begin{equation}\label{BlackBox1}
\max_{\l\in J} |\pwt(\l)| \leq C_1 m^{-\alpha|\omega\wedge\tau|} \text{ for all }\omega,\tau\in\Omega_\varepsilon, \omega\neq\tau,
\end{equation}
\begin{equation}\label{BlackBox2}
\sup_{v\in\R} Leb(\lbrace \l\in J: |v+\pwt(\l)|\leq r\rbrace )\leq C_2 m^{|\omega\wedge\tau|\beta}r
\end{equation}
for all $\omega,\tau\in\Omega_\varepsilon, \omega\neq\tau$ such that $|\omega\wedge\tau| \geq k_0$, and
\begin{equation}\label{BlackBox3}
\max_{u\in \mathcal{A}^{n}, [u]\cap \Omega_\varepsilon \neq \emptyset} \mu([u])\leq C_3m^{-\gamma n} \text{ for all } n\geq 1.
\end{equation}
Then the convolution $\eta\ast\nu_\l$ is absolutely continuous with respect to the Lebesgue measure for a.e. $\l\in J$.
\end{Proposition}

\begin{Remark}
In fact, in Proposition \ref{BlackBox} the condition on exact dimensionality of the measure $\eta$ can be replaced by the following condition (and this is the only consequence of exact dimensionality of $\eta$ that was used in the proof of Proposition \ref{BlackBox} in \cite{DGS}):
 $\eta$ is a compactly supported Borel probability measure on the real line, such that
\begin{equation}\label{e.exdim}
\eta[B_r(x)] \leq Cr^{d_\eta}, \text{  for all }x\in\R \text{  and  } r>0.
\end{equation}
\end{Remark}

\section{Proofs}\label{sec:main}
Here we construct the measure $\eta$ supported on $K$ and a family of measures $\nu_\lambda$ with $supp\, \nu_\l=C_\l$ such that Proposition \ref{BlackBox} can be applied. Since absolute continuity of the convolution $\eta * \nu_\l$ implies that $Leb(C_\l+K)>0$, this will prove Theorem \ref{t.1}.

Let us start with construction of the measure $\eta$. The compact set $K\subset \mathbb{R}$ satisfies the condition (\ref{e.condition}), i.e. $\dim_H(K)+\dim_H(C_{\lambda}) > 1\ \text{ for all }\ \l\in J$. Take any constant $d\in (0, \dim_H(K))$ that is sufficiently close to $\dim_H(K)$ to guarantee that  $d+\dim_H(C_{\lambda}) > 1\ \text{ for all }\ \l\in J$.  By Frostman's Lemma (see, e.g., \cite[Theorem~8.8]{Mattila}), there exists a Borel measure $\eta$ supported on $K$ such that (\ref{e.exdim}) holds with $d_\eta=d$.

Let $\{C_\lambda\}_{\l\in J}$ be a family  of dynamically defined Cantor sets generated by contractions $f_{k, \lambda}:[0,1]\to [0,1]$, $k=1, \ldots, m$, given by (\ref{e.start}). Define the map $\xi: C_\lambda\to\mathbb{R}$ by
$$
\xi(x)=\log|f_{k, \lambda}'(f_{k, \lambda}^{-1}(x))|\ \ \text{if}\ \ x\in f_{k, \lambda}([0,1]).
$$
Due to \cite{Man}, there is an ergodic Borel probability measure $\mu_\l$ on $C_\l$ (namely, the equilibrium measure for the potential $(\text{dim}_H\,C_{\lambda})\xi(x)$)  that satisfies the condition $-h_{\mu_\l}/\mu_\l (\xi) = \dim_H (C_\l)$.  This is also a measure on $C_\l$ such that $\dim_H(\mu_\l) = \dim_H (C_\l)$ (i.e. {\it the measure of maximal dimension}).

Sometimes it is convenient to consider one expanding map $$\Phi_\lambda:\cup_{k=1}^m
f_{k, \lambda}([0,1])\mapsto [0,1], \ \text{where}\ \Phi_\lambda(x)=f^{-1}_{k,\lambda}(x)\ \text{ for}\  x\in f_{k, \lambda}([0,1]),$$  instead of the collection of contractions $\{f_{1, \lambda}, \ldots, f_{m,\lambda}\}$. Notice that $\Phi_\lambda(C_\lambda)=C_\lambda$, and the Lyapunov exponent of $\Phi_\lambda$ with respect to the invariant measure $\mu_\lambda$ is equal to $-\mu_\l (\xi)$. We will denote this Lyapunov exponent by $Lyap^u(\mu_\lambda)$. Since $\mu_\l$ is a measure of maximal dimension, we have
$$
\dim_H(\mu_\l) = \frac{h_{\mu_\l}(\Phi_\l)}{Lyap^u(\mu_\l)}=\dim_H(C_\l).
$$

 For each $\omega\in\Omega$ let $F_{\l}^n(\omega) = f_{\omega_0,\l}\circ\dots\circ f_{\omega_n,\l}(x)$, for a fixed $x\in [0,1]$.  Then we can define the map $\Pi_\lambda : \Omega \to \R$ given by
\begin{equation}
\Pi_\lambda (\omega) = \lim_{n\to\infty} F_{\l}^n(\omega),
\end{equation}
where in fact the limit does not depend on the initial point $x\in [0,1]$. For any $\l_1, \l_2\in J$ the map $h_{\l_1, \l_2}:C_{\l_2}\to C_{\l_1}$ defined by $h_{\l_1, \l_2}=\Pi_{\l_1}\circ \Pi^{-1}_{\l_2}$ is a homeomorphism. It is well known (e.g. see Section 19 in \cite{KH}) that this homeomorphism must be H\"older continuous. Moreover, due to \cite{PV} the following statement holds.

\begin{Lemma}\label{l.help1}
For any $\lambda_0\in J$ and any $\tau\in (0,1)$ there exists a neighborhood $V\subseteq J$, $\lambda_0\in V$, such that for any $\l\in V$ the conjugacy $h_{\l, \l_0}:C_{\l_0}\to C_{\l}$ as well as its inverse $h_{\l_0, \l}:C_{\l}\to C_{\l_0}$ are H\"older continuous with H\"older exponent $\tau$.
\end{Lemma}

Define the measure $\mu$ on $\Omega$ by $\mu:=\Pi_{\l_0}^{-1}(\mu_{\l_0})$, and set $$\nu_\lambda:=\Pi_{\l}(\mu)=\Pi_{\l}(\Pi^{-1}_{\l_0}(\mu_{\l_0}))=h_{\l, \l_0}(\mu_{\l_0})=h_{\l, \l_0}(\nu_{\l_0}).$$ If both $h_{\l, \l_0}$ and $h_{\l_0, \l}$ are H\"older continuous with H\"older exponent $\tau$ then
$$
\tau\dim_HC_{\l_0}=\tau\dim_H\nu_{\l_0}\le \dim_H\nu_\l\le \frac{1}{\tau}\dim_H\nu_{\l_0}=\frac{1}{\tau}\dim_HC_{\l_0}.
$$
Since in Lemma \ref{l.help1} the value of $\tau$ can be taken arbitrarily close to one, we get the following statement.

\begin{Lemma}\label{l.help2}
For any $\lambda_0\in J$  there exists a neighborhood $W\subseteq J$, $\lambda_0\in W$, such that for any $\l\in W$ we have
$$
d_\eta + \dim_H \nu_\l > 1.
$$
\end{Lemma}

It is clear that in order to prove Theorem \ref{t.1} it is enough to prove that for each $\l_0\in J$ there exists a neighborhood $W, \l_0\in W$, such that the sum $K+C_{\l}$ has positive Lebesgue measure for a.e. $\l$ from $W$. For a given $\l_0\in J$ we can choose positive $\varepsilon,\alpha, \beta$, and $\gamma$ in such a way that
$$
d_\eta + \dim_H \nu_{\l_0} > 1+\varepsilon,
$$
$$\alpha < \frac{Lyap^u(\nu_{\l_0})}{\log m} < \beta,$$
$$\gamma < \frac{h_{\nu_{\l_0}}(\Phi_{\l_0})}{\log m}.$$
If $\alpha, \beta$ are sufficiently close to $\frac{Lyap^u(\nu_{\l_0})}{\log m}$, and $\gamma$ is sufficiently close to $\frac{h_{\nu_{\l_0}}(\Phi_{\l_0})}{\log m}$, then we also have
$$
d_\eta + \frac{\gamma}{\beta} > 1,
$$
which is one of the conditions (\ref{BlackBox0}) of Proposition \ref{BlackBox}, and also
$$\frac{\beta}{\alpha} < 1+\frac{\varepsilon}{2}.$$
Decreasing if needed the neighborhood $W$ given by Lemma \ref{l.help2} we can guarantee that for all $\l\in W$ the following property holds:
$$\frac{h_{\nu_\l}(\Phi_\l)}{Lyap^u(\nu_\l)} -\frac{\gamma}{\alpha} < \frac{\varepsilon}{2}.$$
Therefore
$$d_\eta > 1+\varepsilon - \dim_H(\nu_\l) > \frac{\beta}{\alpha} - \frac{\gamma}{\alpha}$$
for $\l\in W$, which implies another part of the condition (\ref{BlackBox0}) of Proposition \ref{BlackBox}, namely,
$$d_\eta > \frac{\beta - \gamma}{\alpha}.$$
Finally let us notice that if $W$ is small, then we have
$$\alpha < \frac{Lyap^u(\nu_\l)}{\log m} < \beta,$$
$$\gamma < \frac{h_{\nu_\l}(\Phi_\l)}{\log m}$$
for all $\l\in W$.

In order to verify the conditions (\ref{BlackBox1}), (\ref{BlackBox2}), and (\ref{BlackBox3}) of Proposition \ref{BlackBox}, we will try to mimic the proof of Theorem 3.7 from \cite{DGS}. We will show that for a given small $\varepsilon>0$ there are subsets $\Omega_1$ and $\Omega_2$ in $\Omega$ such that $\mu(\Omega_i)>1-\frac{\varepsilon}{2}$, $i=1,2$, and properties (\ref{BlackBox1}) and (\ref{BlackBox2}) hold for all $\omega,\tau\in\Omega_1$, and (\ref{BlackBox3}) holds for all $\omega,\tau\in\Omega_2$. This will imply that all these conditions hold for all  $\omega,\tau\in\Omega_{\varepsilon}=\Omega_1\cap\Omega_2$ with $\mu(\Omega_\varepsilon)>1-\varepsilon$, i.e. justify application of Proposition \ref{BlackBox}, and therefore prove Theorem \ref{t.1}.

 For $\omega\in \Omega$, $\omega=\omega_0\omega_1\ldots\omega_n\ldots$, set $p(\lambda)=\Pi_\lambda(\omega)$ and
 \begin{equation}\label{e.ls}
 l^{(s)}=\frac{df_{\omega_{s-1}, \lambda}}{dx}(\Phi^s_{\lambda}(p(\lambda))).
 \end{equation}
 We will also write $l^{(s)}(\lambda)$ or $l^{(s)}_\omega$ if we need to emphasize the dependence of $l^{(s)}$ on $\lambda$ or $\omega$. Notice that $\{l^{(s)}\}$ is a sequence of multipliers of the contractions along the orbit of point $p(\lambda)$ under the map $\Phi_\lambda$, and if Lyapunov exponent at $p(\lambda)$ exists then
 $$
 Lyap^u(p(\lambda))=-\lim_{n\to \infty} \frac{1}{n}\sum_{s=1}^n\log\left|l^{(s)}\right|.
 $$
 \begin{Lemma}\label{Egorov}
Given $\epsilon > 0$, there exists a set $\Omega_1\subset \Omega$ with $\mu(\Omega_1) > 1-\frac{\epsilon}{2}$ and $N\in\N$ such that
 \begin{equation}\label{e.alphabeta}\alpha\log m < -\frac{1}{n}\sum_{s=1}^{n} \log \left|l^{(s)}(\l)\right|<\beta\log m \end{equation}
for every $\l\in W$, $n\ge N$, and all $p\in \Pi_\l(\Omega_1)$.
\end{Lemma}
\begin{proof}[Proof of Lemma \ref{Egorov}]
Let us start with the first part of the inequality (\ref{e.alphabeta}). First we will show that for a fixed $\l\in W$ and a given $\varepsilon'>0$, there exists $\Omega'$ with $\mu (\Omega') > 1-\varepsilon'$ and $N\in \N$ such that
$$\alpha\log m+\xi < -\frac{1}{n}\sum_{s=1}^{n} \log \left|l^{(s)}(\l)\right|,$$
 where $0 < \xi < Lyap^u(\mu_\l) - \alpha\log m$, for all $n\ge N$ and all $p\in \Pi_\l(\Omega')$.

By the Birkhoff Ergodic Theorem,
\begin{align*}
Lyap^u(\mu_\l) & =\\
& = \int \log \| D\Phi_\l (\Pi_\l(\omega)\|\,d\mu(\omega) \\
& = \lim_{n\to\infty} \frac{1}{n}\sum_{s=1}^n \log \| D\Phi_\l(\Phi_\l^s(\Pi_\l(\omega))\| \\
& = \lim_{n\to\infty}-\frac{1}{n}\sum_{s=1}^n \log \left|l_\omega^{(s)}(\l)\right|
\end{align*}
for $\mu$-a.e. $\omega\in\Omega$.  Thus by Egorov's theorem, there exists $\Omega' \subset \Omega$ with $\mu(\Omega') > 1-\varepsilon'$ such that the convergence is uniform on $\Omega'$.  Thus there exists $N\in \N$ such that $\alpha\log m+\xi < -\frac{1}{n}\sum_{s=1}^{n} \log \left|l^{(s)}(\l)\right|$ for all $n\ge N$ and all $p\in \Pi_\l(\Omega')$.

Next we will show that $N$ can be chosen uniformly in  $\l \in W$.  Let $\varepsilon >0$ be given.  Consider the family of functions
$$L_\omega (\l) = -\log \|D\Phi_\l(\Pi_\l(\omega)\|.$$
We can treat the elements of this family as functions of $\l$ with parameter $\omega$.  Then $\lbrace L_\omega (\l)\rbrace_{\omega\in\Omega}$ is an equicontinuous family of functions and there exists $t>0$ such that if $|\l_1 - \l_2 | \leq t$, then $|L_\omega(\l_1) - L_\omega(\l_2)| < \frac{\xi}{100}$ for any $\omega\in\Omega.$ Consider a finite $t$-net $\lbrace y_1,\dots,y_M\rbrace$ in $W$, containing $M=M(W,t)$ points.  For each point $y_j$ we can find a set $\Omega^{(j)}\subset \Omega$, $\mu(\Omega^{(j)}) > 1 - \frac{\epsilon}{4M}$, and $N_j\in \N$ such that for every $n \geq N_j$ and every $\omega\in \Omega^{(j)}$, we have
$$
\frac{1}{n}\sum_{s=1}^{n} L_{\sigma^s(\omega)}(y_j) = -\frac{1}{n}\sum_{s=1}^{n} \log \left|l_\omega^{(s)}(y_j)\right| > \alpha\log m +\xi.
$$
Take $\Omega_1 = \cap_{j=1}^M \Omega^{(j)}$.  We have
$$\mu (\Omega_1) > 1 - M\frac{\varepsilon}{4M} = 1-\frac{\varepsilon}{4},$$
and for every $\l\in W$ there exists $y_j$ with $|y_j-\l| \leq t$.  So for every $\omega\in\Omega_1\subseteq\Omega^{(j)}$ and every $n \ge N =\max \lbrace N_1,\dots,N_M\rbrace$, we have
\begin{align*}
-\frac{1}{n}\sum_{s=1}^{n} \log \left|l_\omega^{(s)}(\l)\right| & = \frac{1}{n}\sum_{s=1}^{n} L_{\sigma^s(\omega)}(\l)\\
& \geq \frac{1}{n}\sum_{s=1}^{n} L_{\sigma^s(\omega)}(y_j) - \left|\frac{1}{n}\sum_{s=1}^{n} L_{\sigma^s(\omega)}(y_j) - \frac{1}{n}\sum_{s=1}^{n} L_{\sigma^s(\omega)}(\l)\right| \\
& \geq \alpha\log m + \xi -\frac{\xi}{100} \\
& > \alpha\log m +\frac{\xi}{2} \\
& > \alpha\log m,
\end{align*}
which  proofs the first part of the inequality (\ref{e.alphabeta}). The proof of the second part is analogous. This concludes the proof of Lemma \ref{Egorov}.
\end{proof}
Notice that Lemma \ref{Egorov} directly implies that for $p\in \Pi_\l(\Omega_1)$ and $n\ge N$ we have
\begin{equation}\label{e.abproduct}  m^{-\beta n} < \left|\prod_{s=1}^n l^{(s)}\right| < m^{-\alpha n}.  \end{equation}
The next statement is a simple partial case of Lemma 3.12 from \cite{DGS}.
\begin{Lemma}\label{l.333simple}
There is a constant $C >0$ such that for any word $\omega_0\omega_1\ldots\omega_n\in\mathcal{A}^{n+1}$, any $\lambda\in J$, and any $x,y\in I_0=[0,1]$ the following holds. Set $$p=f_{\omega_0, \lambda}\circ f_{\omega_1, \lambda}\circ\ldots \circ f_{\omega_n, \lambda}(x),$$ and define $\{l^{(s)}\}$ by (\ref{e.ls}). Denote $$q=f_{\omega_0, \lambda}\circ f_{\omega_1, \lambda}\circ\ldots \circ f_{\omega_n, \lambda}(y).$$ Then
$$\frac{1}{C}\left|\prod_{s=1}^n l^{(s)}\right| \leq \frac{|p-q|}{|x-y|} \leq C\left|\prod_{s=1}^n l^{(s)}\right|. $$
\end{Lemma}
The property (\ref{BlackBox1}) for all $\omega, \tau\in \Omega_1$ follows now from (\ref{e.abproduct}) and Lemma \ref{l.333simple}.

In order to check (\ref{BlackBox2}) for some $\omega, \tau\in \Omega$ it is enough to show that
\begin{equation}\label{e.ineq}
\left|\frac{d}{d\lambda}\phi_{\omega, \tau}(\lambda)\right|>C' m^{-\beta |\omega \wedge \tau|}
\end{equation}
 for some uniform (independent of $\omega, \tau\in \Omega$) constant $C'>0$.

Let us consider some $\omega, \tau\in \Omega$, and set $n=|\omega \wedge \tau|$. Let us denote
$$
P_0(\lambda)=\Pi_{\lambda}(\omega),  \ Q_0(\lambda)=\Pi_{\lambda}(\tau), \ \
$$
and
$$
\ P_s=\Phi^s_{\lambda}(P_0),  \ Q_s=\Phi_{\lambda}^s(Q_0),  \ s=0, \ldots, n.
$$
Notice that the distance between $P_n$ and $Q_n$ is uniformly bounded away from zero. Indeed, since $n=|\omega \wedge \tau|$, $P_n$ and $Q_n$ belong to different elements of Markov partition of $C_\lambda$. Let us also denote
$$
k^{(s)}_\lambda=f_{\omega_{s-1}} \ \ \ \text{and}\ \ \  l^{(s)}(\lambda)=\frac{\partial k^{(s)}_\lambda}{\partial x}(P_{s}(\lambda))\ \ \text{for}\ \ s=1, 2, \ldots, n.
$$
We have
\begin{equation}\label{e.ineqnext}
k^{(s)}_\lambda(x)=k^{(s)}_\lambda(P_s(\lambda))+l^{(s)}\cdot (x-P_s(\lambda)) + O((x-P_s(\lambda))^2)
\end{equation}
and
\begin{equation}\label{e.ineqnext}
\frac{\partial k^{(s)}_\lambda}{\partial x}(x)=l^{(s)}+  O(x-P_s(\lambda)).
\end{equation}
Notice that $P_0=k^{(1)}_\lambda\circ \ldots    \circ k^{(n)}_\lambda(P_n)$, and $Q_0=k^{(1)}_\lambda\circ \ldots    \circ k^{(n)}_\lambda(Q_n)$.

To prove (\ref{e.ineq}) we need to find a bound on
\begin{align*}\frac{d}{d\lambda}\phi_{\omega, \tau}(\lambda)=\frac{d}{d\l}(P_0(\l)-Q_0(\l))=
\end{align*}
\begin{align*}
  & = \sum_{i=1}^n \left(\prod_{s=1}^{i-1}\frac{\partial k^{(s)}_\l}{\partial x}(P_{s}(\l))\right)\frac{\partial k^{(i)}_\l}{\partial\l}(P_{i}(\l)) + \left(\prod_{s=1}^n \frac{\partial k^{(s)}_\l}{\partial x}(P_{s}(\l))\right)\frac{\partial P_n}{\partial\l}(\l) \\
  & - \sum_{i=1}^n \left(\prod_{s=1}^{i-1}\frac{\partial k^{(s)}_\l}{\partial x}(Q_{s}(\l))\right)\frac{\partial k^{(i)}_\l}{\partial\l}(Q_{i}(\l)) - \left(\prod_{s=1}^n \frac{\partial k^{(s)}_\l}{\partial x}(Q_{s}(\l))\right)\frac{\partial Q_n(\l)}{\partial\l} \\
 & = \sum_{i=1}^n \left(\prod_{s=1}^{i-1}\frac{\partial k^{(s)}_\l}{\partial x}(P_{s}(\l))\right)\left(\frac{\partial k^{(i)}_\l}{\partial\l}(P_{i}(\l))-\frac{\partial k^{(i)}_\l}{\partial\l}(Q_{i}(\l))\right) \\
 & + \sum_{i=1}^n \frac{\partial k^{(i)}_\l}{\partial \l}(Q_{i}(\l))\left(\prod_{s=1}^{i-1}\frac{\partial k^{(s)}_\l}{\partial x}(P_{s}(\l))-\prod_{s=1}^{i-1}\frac{\partial k^{(s)}_\l}{\partial x}(Q_{s}(\l))\right) \\
 & + \left(\left(\prod_{s=1}^n \frac{\partial k^{(s)}_\l}{\partial x}(P_{s}(\l))\right)\frac{\partial P_n}{\partial\l}(\l) - \left(\prod_{s=1}^n \frac{\partial k^{(s)}_\l}{\partial x}(Q_{s}(\l))\right)\frac{\partial Q_n(\l)}{\partial\l}\right)\\
 & = S_1 + S_2 + S_3
  \end{align*}

  Let us estimate $S_1$. We have
\begin{align*}S_1 = & \sum_{i=1}^n \left(\prod_{s=1}^{i-1}\frac{\partial k^{(s)}_\l}{\partial x}(P_{s}(\l))\right)\left(\frac{\partial k^{(i)}_\l}{\partial\l}(P_{i}(\l))-\frac{\partial k^{(i)}_\l}{\partial\l}(Q_{i}(\l))\right)\\
 & = \sum_{i=1}^n \left(\prod_{s=1}^{i-1} l^{(s)}\right)\frac{\partial^2 k^{(i)}_\l}{\partial x\partial\l}(W_{i}(\l))(P_{i}(\l) - Q_{i}(\l))
\end{align*}
where $W_{i}(\l)$ is a point between $P_{i}(\l)$ and $Q_{i}(\l)$.

Since we have
$$
\frac{\partial^2 k^{(i)}_\l}{\partial x\partial\l}=\frac{\partial c_{\omega_{i-1}}}{\partial\l}+\frac{\partial^2 g_{\omega_{i-1}}(x,\l)}{\partial x\partial\l},
$$
 the assumption (\ref{e.c}) implies that $\frac{\partial^2 k^{(i)}_\l}{\partial x\partial\l}$ has sign opposite to the sign of $l^{(i)}$.
Also, it is easy to see that $P_{i}(\l) - Q_{i}(\l)$ has the same sign as $$\left(\prod_{s=i+1}^n l^{(s)}\right)(P_{n}(\l) - Q_{n}(\l)).$$ Therefore all terms in the sum $S_1$ have the same sign as $$-\left(\prod_{s=1}^n l^{(s)}\right)(P_{n}(\l) - Q_{n}(\l)).$$ Using Lemma \ref{l.333simple}, assumption (\ref{e.c}), and the fact that $|P_{n}(\l) - Q_{n}(\l)|$ is bounded away from zero, this implies that
\begin{align}\label{e.last}
|S_1| =  &  \sum_{i=1}^n \left|\prod_{s=1}^{i-1} l^{(s)}\right|\left|\frac{\partial^2 k^{(i)}_\l}{\partial x\partial\l}(W_{i}(\l))\right|\left|P_{i}(\l) - Q_{i}(\l)\right|\ge
  nC^{*}\left|\prod_{s=1}^{n} l^{(s)}\right|
\end{align}
for some constant $C^*>0$.

Let us now estimate  $S_2$. Let us remind that  $k_{\l}^{(s)}(x) = f_{\omega_{s-1}, \l} (x)= c_{\omega_{s-1}}({\l})x+b_{\omega_{s-1}}{(\l)} + g_{\omega_{s-1}}{(x, \l)},$ where the $C^2$-norm of $g_{\omega_{s-1}}{(x,\l)}$ is small.
\begin{align*}
|S_2|
 &= \left|\sum_{i=1}^n \frac{\partial k^{(i)}_\l}{\partial \l}(Q_{i}(\l))\left(\prod_{s=1}^{i-1}\frac{\partial k^{(s)}_\l}{\partial x}(P_{s}(\l))-\prod_{s=1}^{i-1}\frac{\partial k^{(s)}_\l}{\partial x}(Q_{s}(\l))\right)\right| \\
& \leq \sum_{i=1}^n \left|\frac{\partial k^{(i)}_\l}{\partial \l}(Q_{i}(\l))\right|\cdot\left|\left(\prod_{s=1}^{i-1}\frac{\partial k^{(s)}_\l}{\partial x}(P_{s}(\l))-\prod_{s=1}^{i-1}\frac{\partial k^{(s)}_\l}{\partial x}(Q_{s}(\l))\right)\right| \\
& \leq \sum_{i=1}^n C\left|\left(\prod_{s=1}^{i-1}\frac{\partial k^{(s)}_\l}{\partial x}(P_{s}(\l))-\prod_{s=1}^{i-1}\frac{\partial k^{(s)}_\l}{\partial x}(Q_{s}(\l))\right)\right| \\
& = C \sum_{i=1}^n \left|\prod_{s=1}^{i-1} l^{(s)}\right| \left|1-\frac{\prod_{s=1}^{i-1}\frac{\partial k^{(s)}_\l}{\partial x}(Q_{s}(\l))}{\prod_{s=1}^{i-1}\frac{\partial k^{(s)}_\l}{\partial x}(P_{s}(\l))}\right|
\end{align*}
\begin{Lemma}\label{l.33}
$$\left|1-\frac{\prod_{s=1}^{i-1}\frac{\partial k^{(s)}_\l}{\partial x}(Q_{s}(\l))}{\prod_{s=1}^{i-1}\frac{\partial k^{(s)}_\l}{\partial x}(P_{s}(\l))}\right| \leq C^{\prime}\left|\prod_{s=i}^{n} l^{(s)}\right|\max_{t=1,\dots,m}\| g_t{(x,\l)}\|_{C^2} $$
for some $C^{\prime} > 0$.
\end{Lemma}
\begin{proof}[Proof of Lemma \ref{l.33}]
Note that if $A$ is near 1 and $B$ is much smaller than 1, we have that $$|\log A| < B \text{ implies }|A-1| \leq 2B.$$ Indeed,
\begin{align*}
|\log A| < B & \Rightarrow e^{-B}-1 < A-1 < e^B-1 \\
& \Rightarrow -B+O(B^2) < A-1 < B+O(B^2)\\
& \Rightarrow |A-1| < 2B
\end{align*}
for small $B$.

To prove Lemma \ref{l.33}, we will show that $\left| \log \frac{\prod_{s=1}^{i-1}\frac{\partial k^{(s)}_\l}{\partial x}(Q_{s}(\l))}{\prod_{s=1}^{i-1}\frac{\partial k^{(s)}_\l}{\partial x}(P_{s}(\l))}\right|$ is small.  By the mean value theorem and using Lemma \ref{l.333simple} we get
\begin{align*}
\left| \log \frac{\prod_{s=1}^{i-1}\frac{\partial k^{(s)}_\l}{\partial x}(Q_{s}(\l))}{\prod_{s=1}^{i-1}\frac{\partial k^{(s)}_\l}{\partial x}(P_{s}(\l))}\right|
& = \left|\sum_{s=1}^{i-1} \log \frac{\partial k^{(s)}_\l}{\partial x}(Q_{s}(\l)) - \sum_{s=1}^{i-1} \log \frac{\partial k^{(s)}_\l}{\partial x}(P_{s}(\l))\right| \\
& \leq C \sum_{s=1}^{i-1} \left|\frac{\partial k^{(s)}_\l}{\partial x}(Q_{s}(\l))-\frac{\partial k^{(s)}_\l}{\partial x}(P_{s}(\l))\right| \\
& = C\sum_{s=1}^{i-1} \left| \frac{\partial g_{\omega_{s-1}}}{\partial x}(Q_{s}(\l)) - \frac{\partial g_{\omega_{s-1}}}{\partial x}(P_{s}(\l))\right| \\
& = C\sum_{s=1}^{i-1} \left|\frac{\partial^2 g_{\omega_{s-1}}}{\partial x^2}(V_{s+1})\right|\left|Q_{s}(\l) - P_{s}(\l)\right| \\
& \leq \widetilde{C}\max_{t=1,\dots,m}\| g_t{(x,\l)}\|_{C^2} \sum_{s=1}^{i-1} \left|\prod_{j=s+1}^{n} l^{(s)}\right| \\
& \leq \widetilde{\widetilde{C}}\max_{t=1,\dots,m}\| g_t{(x,\l)}\|_{C^2} \left| \prod_{s=i}^n l^{(s)}\right|
\end{align*}
since the terms of the last sum are bounded by a geometrical progression.  This proves  Lemma \ref{l.33}.
\end{proof}
Therefore we have
\begin{equation}\label{e.eqnew1}
|S_2| \leq n{C^{\prime\prime}}\left|\prod_{s=1}^n l^{(s)}\right|\max_{t=1,\dots,m}\| g_t{(x,\l)}\|_{C^2}.
\end{equation}
Notice that Lemma \ref{l.33} implies also that for come constant $\hat{C}>0$ we have
\begin{equation}\label{e.eqnew2}
|S_3|\le \hat{C}\left|\prod_{s=1}^n l^{(s)}\right|.
\end{equation}

Now combining (\ref{e.last}), (\ref{e.eqnew1}), and (\ref{e.eqnew2}) we get

\begin{align*}
\left| \frac{d}{d\lambda}\phi_{\omega, \tau}(\lambda)\right|=|S_1+S_2+S_3|\ge \left(nC^{*}-n{C^{\prime\prime}}\max_{t=1,\dots,m}\| g_t{(x,\l)}\|_{C^2}-\hat{C}\right)\left|\prod_{s=1}^{n} l^{(s)}\right|.
\end{align*}
Therefore one can choose smallness of the $C^2$ norms of $\{g_i\}_{i=1, \ldots, m}$ in (\ref{e.g}) so that for some $\delta^*>0$ and all large enough values of $n\in \mathbb{N}$ we have
\begin{align}\label{e.estfinal}
\left| \frac{d}{d\lambda}\phi_{\omega, \tau}(\lambda)\right|\ge n\delta^*\left|\prod_{s=1}^{n} l^{(s)}\right|
\end{align}
for any $\omega, \tau\in \Omega$ with $|\omega\wedge\tau|=n$. In particular, if $\omega, \tau\in \Omega_1$ then (\ref{e.estfinal}) together with (\ref{e.abproduct}) imply that
$$
\left| \frac{d}{d\lambda}\phi_{\omega, \tau}(\lambda)\right|\ge n\delta^*m^{-\beta n}=n\delta^*m^{-\beta |\omega\wedge\tau|},
$$
which implies (\ref{e.ineq}) and hence verifies the assumption (\ref{BlackBox2}).

Finally, the Shannon-McMillan-Breiman Theorem implies that
$$-\frac{1}{n}\log \mu([\omega]_n)\to h_{\mu}(\sigma)$$
for $\mu$-a.e. $\omega \in \Omega$.  By Egorov's theorem, there exists a set $\Omega_2\subset \Omega$ with $\mu(\Omega_2) > 1-\varepsilon/2$  such that this convergence is uniform in $\omega\in \Omega_2$.  Thus we have
$$-\frac{1}{n}\log \mu([\omega]_n)\to h_{\mu}(\sigma) > \gamma\log m $$
uniformly for $\omega\in\Omega_2$.  So for $n$ sufficiently large, we have that
$$\mu([\omega]_n) < m^{-\gamma n}.$$
Hence if $C>0$ is sufficiently large then for all $n\geq 1$ we have
$$\mu([\omega]_n) < Cm^{-\gamma n}.$$

Now let $\Omega_\varepsilon = \Omega_1 \cap \Omega_2$, then $\mu(\Omega_\varepsilon) > 1-\varepsilon$ and all conditions of Proposition \ref{BlackBox} hold on $\Omega_\varepsilon$. This concludes the proof of Theorem \ref{t.1}.


\begin{thebibliography}{00}


\bibitem{BPS} B. B\'ar\'any, M. Pollicott, K. Simon, Stationary measures for projective transformations: the Blackwell and F\"urstenberg measures, \textit{J.\ Stat.\ Phys.}\ \textbf{148} (2012), 393--421.

\bibitem{BM} G.\ Brown,  W.\ Moran, Raikov systems and radicals in convolution measure algebras, {\it J. London Math. Soc. (2)} {\bf  28} (1983), no. 3, pp. 531--542.

\bibitem{BKMP} G.\ Brown, M.\ Keane, W.\ Moran, C.\ Pearce,  An inequality, with applications to Cantor measures and normal numbers, {\it Mathematika} {\bf 35} (1988), no. 1, pp. 87--94.








\bibitem{CF} T.\ Cusick, M.\ Flahive, The Markoff and Lagrange spectra,{\it Mathematical Surveys and Monographs}, {\bf  30}, American Mathematical Society, Providence, RI, 1989.


\bibitem{D05} D.\ Damanik, Dynamical upper bounds for one-dimensional quasicrystals, \textit{J.\ Math.\ Anal.\ Appl.}\ \textbf{303} (2005), 327--341.

\bibitem{DEG} D.\ Damanik, M.\ Embree, A.\ Gorodetski, Spectral properties of Schr\"odinger operators arising in the study of quasicrystals, chapter in \textit{Mathematics of Aperiodic Order} (editors Johannes Kellendonk, Daniel Lenz, Jean Savinien), series \textit{Progress in Mathematics, Birkh\"aeuser}, \textbf{309} (2015), 307--370. 


\bibitem{DG09} D.\ Damanik, A.\ Gorodetski, Hyperbolicity of the trace map for the weakly coupled Fibonacci Hamiltonian, \textit{Nonlinearity} \textbf{22} (2009), 123--143.

\bibitem{DG11} D.\ Damanik, A.\ Gorodetski, Spectral and quantum dynamical properties of the weakly coupled Fibonacci Hamiltonian, \textit{Commun.\ Math.\ Phys.}\ \textbf{305} (2011), 221–-277.


\bibitem{DG12} D.\ Damanik, A.\ Gorodetski, The density of states measure of the weakly coupled Fibonacci Hamiltonian, \textit{Geom.\ Funct.\ Anal.}\ \textbf{22} (2012), 976–-989.


\bibitem{DGS} D.\ Damanik, A.\ Gorodetski, B.\ Solomyak,  Absolutely Continuous Convolutions of Singular Measures and an Application to the Square Fibonacci Hamiltonian, {\it Duke Mathematical Journal}\ \textbf{164} (2015), 1603--1640.

\bibitem{DGY14} D.\ Damanik, A.\ Gorodetski, W.\ Yessen, The Fibonacci Hamiltonian, preprint (arXiv:1403.7823).









\bibitem{EL06} S.\ Even-Dar Mandel, R.\ Lifshitz, Electronic energy spectra and wave functions on the square Fibonacci tiling, \textit{Phil.\ Mag.}\ \textbf{86} (2006), 759--764.

\bibitem{EL07} S.\ Even-Dar Mandel, R.\ Lifshitz, Electronic energy spectra of square and cubic Fibonacci quasicrystals, \textit{Phil.\ Mag.}\ \textbf{88} (2008), 2261--2273.

\bibitem{EL08} S.\ Even-Dar Mandel, R.\ Lifshitz, Bloch-like electronic wave functions in two-dimensional quasicrystals, preprint (arXiv:0808.3659).


\bibitem{Gar} I.\ Garcia, A family of smooth Cantor sets, \textit{Ann.\ Acad.\ Sci.\ Fenn.\ Math.}\ \textbf{36} (2011), 21--45.

\bibitem{H} M.\ Hall, On the sum and product of continued fractions, {\it Ann. of Math. (2)}, {\bf 48} (1947), pp. 966--993.

\bibitem{HS} M.\ Hochman, P.\ Shmerkin, Local entropy averages and projections of fractal measures, \textit{Ann.\ of Math.}\ \textbf{175} (2012), 1001-–1059.

\bibitem{HMP} B.\ Honary, C.\ Moreira, M.\ Pourbarat, Stable intersections of affine Cantor sets, \textit{Bull. Braz. Math. Soc.} \textbf{36} (2005), 363--378.

\bibitem{ILML} R.\ Ilan, E.\ Liberty, S.\ Even-Dar Mandel, R.\ Lifshitz, Electrons and phonons on the square Fibonacci tilings, \textit{Ferroelectrics} \textbf{305} (2004), 15--19.


\bibitem{KH}  A.\ Katok, B.\ Hasselblatt, \textit{Introduction to the Modern Theory of Dynamical Systems}, Cambridge University Press, 1995.



\bibitem{L} R.\ Lifshitz, The square Fibonacci tiling, \textit{J. of Alloys and Compounds}, \textbf{342} (2002), 186--190.







\bibitem{Mal} A.\ Malyshev, Markov and Lagrange spectra (a survey of the literature), {\it Zap. Nauchn. Sem. Leningrad. Otdel. Mat. Inst. Steklov.} , {\bf 67} (1977) pp. 5--38 (in Russian).

\bibitem{Man} A.\ Manning, A relation between Lyapunov exponents, Hausdorff dimension and entropy, \textit{Ergodic Theory Dynam.\ Systems} \textbf{1} (1981), 451–-459.

\bibitem{Mattila} P.\ Mattila, {\em Geometry of Sets and Measures in Euclidean Spaces}, Cambridge University Press, Cambridge, 1995.

\bibitem{MM} H.\ McCluskey, A.\ Manning, Hausdorff dimension for horseshoes, \textit{Ergodic Theory Dynam.\ Systems} \textbf{3} (1983), 251-–260.

\bibitem{M} M.\ Mei, Spectra of discrete Schr\"odinger operators with primitive invertible substitution potentials, \textit{J. Math. Phys.} \textbf{55} (2014), no. 8, 082701, 22pp.

\bibitem{MO} P.\ Mendes, F.\ Oliveira, On the topological structure of the arithmetic sum of two Cantor sets, \textit{Nonlinearity} \textbf{7} (1994), 329--343.

\bibitem{Moreira} C.\ Moreira, Sums of regular Cantor sets, dynamics and applications to number theory, International Conference on Dimension and Dynamics (Miskolc, 1998), \textit{Period.\ Math.\ Hungar.}\ \textbf{37} (1998), 55-–63.

\bibitem{MY} C.\ Moreira, J.-C.\ Yoccoz,  Stable intersections of regular Cantor sets with large Hausdorff dimensions, \textit{Ann.\ of Math.}\ \textbf{154} (2001), 45--96.

\bibitem{NPS} F.\ Nazarov, Y.\ Peres, P.\ Shmerkin, Convolutions of Cantor measures without resonance, \textit{Israel J.\ Math.}\ \textbf{187} (2012), 93--116.

\bibitem{N} J.\ Neunh\"auserer, Properties of some overlapping self-similar and some self-affine measures, \textit{Acta Math.\ Hungar.}\ \textbf{92} (2001), 143--161.


\bibitem{n1} S.\ Newhouse, Non-density
of Axiom A(a) on $S^2$, {\it Proc. A.M.S. Symp. Pure Math.}, {\bf
14}, (1970), 191--202.

\bibitem{n2}
 S.\ Newhouse, Diffeomorphisms with infinitely many sinks. {\it
Topology} {\bf 13}, (1974), 9--18.

\bibitem{n3}
 S.\ Newhouse, The abundance of wild hyperbolic sets and nonsmooth
stable sets of diffeomorphisms, {\it Publ. Math. I.H.E.S.}, {\bf
50}, (1979), 101--151.

\bibitem{NW} S.-M.\ Ngai, Y. Wang, Self-similar measures associated to IFS with non-uniform contraction ratios, \textit{Asian J.\ Math.}\ \textbf{9} (2005), 227--244.

\bibitem{PaTa} J. Palis, F. Takens, {\em Hyperbolicity and Sensitive Chaotic Dynamics at Homoclinic Bifurcations}, Cambridge University Press, 1993.

\bibitem{PV}  Palis J., Viana M.,  On the continuity of Hausdorff
dimension and limit capacity for horseshoes.  {\it  Lecture Notes in
Math., 1331, Springer, Berlin,} 1988.

\bibitem{PeSch} Y.\ Peres, W.\ Schlag, Smoothness of projections, Bernoulli convolutions, and the dimension of exceptions, \textit{Duke Math.\ J.}\ \textbf{102} (2000), 193--251.

\bibitem{PeShm} Y.\ Peres, P.\ Shmerkin, Resonance between Cantor sets, \textit{Ergodic Theory Dynam.\ Systems} \textbf{29} (2009), 201-–221.

\bibitem{PeSo} Y.\ Peres, B.\ Solomyak, Self-similar measures and intersections of Cantor sets, \textit{Trans.\ Amer.\ Math.\ Soc.}\ \textbf{350} (1998), 4065--4087.

\bibitem{PeSo96} Y.\ Peres, B.\ Solomyak, Absolute continuity of Bernoulli convolutions, a simple proof, \textit{Math.\ Res.\ Lett.}\ \textbf{3} (1996), 231--239.

\bibitem{Po} M.\ Pollicott, Analyticity of dimensions for hyperbolic surface diffeomorphisms, \textit{Proceedings of the American Mathematical Society}\ \textbf{143} (2015), 3465--3474.

\bibitem{PoSi} M.\ Pollicott, K.\ Simon, The Hausdorff dimension of $\lambda$-expansions with deleted digits, \textit{Trans.\ Amer.\ Math. Soc.}\ \textbf{347} (1995), 967--983.




\bibitem{Sae} S.\ Saeki, On convolution squares of singular measures, \textit{Illinois J.\ Math.}\ \textbf{24} (1980), 225-–232.


\bibitem{Sa} A.\ Sannami, An example of a regular Cantor set whose difference set is a Cantor set with positive measure, {\it  Hokkaido Math. J.} {\bf 21} (1992), no. 1, 7--24.

\bibitem{Shm} P.\ Shmerkin,  On the Exceptional Set for Absolute Continuity of Bernoulli Convolutions, {\it Geometric and Functional Analysis}, {\bf 24} (2014),  946--958.

\bibitem{ShmS} P.\ Shmerkin, B.\ Solomyak, Absolute continuity of self-similar measures, their projections and convolutions, preprint (arXiv:1406.0204).



\bibitem{SS} K.\ Simon, B.\ Solomyak, Hausdorff dimension for horseshoes in $\R^3$, \textit{Ergodic Theory Dynam.\ Systems} \textbf{19} (1999), 1343--1363.

\bibitem{SSU} K.\ Simon, B.\ Solomyak, M.\ Urbanski, Invariant measures for parabolic IFS with overlaps and random continued fractions, \textit{Trans.\ Amer.\ Math.\ Soc.}\ \textbf{353} (2001), 5145--5164.





\bibitem{So95} B.\ Solomyak, On the random series $\sum\pm \lambda^n$ (an Erd\H{o}s problem), \textit{Ann.\ of Math.}\ \textbf{142} (1995), 611--625.

\bibitem{So97} B.\ Solomyak, On the measure of arithmetic sums of Cantor sets, \textit{Indag.\ Math.}\ (\textit{N.S.}) \textbf{8} (1997), 133--141.

\bibitem{So98} B.\ Solomyak, Measure and dimension for some fractal families, \textit{Math.\ Proc.\ Cambridge Philos.\ Soc.}\ \textbf{124} (1998), 531--546.




\bibitem{W} N.\ Wiener, A.\ Wintner, Fourier-Stieltjes transforms and singular infinite convolutions, \textit{Amer.\ J.\ Math.}\ \textbf{60} (1938), 513--522.

\bibitem{Y} W.\ Yessen, Hausdorff dimension of the spectrum of the square Fibonacci Hamiltonian, preprint (arXiv:1410.3102).



\end{thebibliography}
\end{document}